\documentclass[reqno]{amsart}

\usepackage[english]{babel}
\usepackage{amsthm, latexsym, enumerate, gastex, amssymb,url,listings,enumitem}

\usepackage[mathcal]{eucal}
\usepackage{graphicx,url}

\copyrightinfo{2015}{Simone Ugolini}

\renewcommand{\phi}[0]{\varphi}
\renewcommand{\theta}[0]{\vartheta}
\renewcommand{\epsilon}[0]{\varepsilon}

\newcommand{\N}{\text{$\mathbf{N}$}}

\newcommand{\R}{\text{$\mathbf{R}$}}

\newtheorem{theorem}{Theorem}[section]
\newtheorem{lemma}[theorem]{Lemma}

\theoremstyle{definition}

\newtheorem{definition}[theorem]{Definition}
\newtheorem{example}[theorem]{Example}

\theoremstyle{remark}
\newtheorem{remark}[theorem]{Remark}

\numberwithin{equation}{section}

\begin{document}

\bibliographystyle{amsplain}

\date{}
\keywords{Digital problems, radix representation, semigroups}
\subjclass[2010]{11A05, 11A63, 20M14}

\title[]
{Multiplicative subsemigroups of the positive integers closed with respect to the number of digits}

\author{S.~Ugolini}
\email{sugolini@gmail.com} 

\begin{abstract}
In this paper we present an alternative approach to a problem dealt with by Rosales et al. In particular, once a base $b$ for the representation of the integers is fixed, we describe a procedure for constructing the smallest multiplicative subsemigroup of the positive integers closed with respect to the number of digits which contains a set of positive integers.
\end{abstract}

\maketitle

\section{Introduction}
In this paper we focus on certain subsemigroups of the multiplicative semigroup $(\N^*, \cdot)$ of the positive integers. 

Let $b$ be a positive integer not smaller than $2$. We say that a subsemigroup $G \subseteq (\N^*, \cdot)$ is a $b$-dc-semigroup if $G$ is closed with respect to the number of digits. This latter is equivalent to saying that if $x \in G$ and $({x_{n-1} x_{n-2} \dots x_0})_{b}$ is its base-$b$ representation with $x_{n-1} \not = 0$, then $G$ contains all the integers  $({y_{n-1} y_{n-2} \dots y_0})_b$ with $y_{n-1} \not = 0$. 

\begin{example}
Let $b =2$ and $G$ a $2$-dc-semigroup. If $101_2 \in G$, then
\begin{equation*}
\{100_2, 101_2, 110_2, 111_2 \} \subseteq G.
\end{equation*}
\end{example}  

Now we introduce some notations, we will use repeatedly along the paper.

\begin{definition}
If $y$ and $z$ are two non-negative integers such that $y < z$, then 
\begin{eqnarray*}
[y, z[_{\N} & = & \{x \in \R: y \leq x < z \} \cap \N; \\
\left[y,+\infty \right[_{\N} & = & \{x \in \R: y \leq x < + \infty \} \cap \N.
\end{eqnarray*}

If $b$ is an integer not smaller than $2$, $i \in \N$ and $j \in \N^*$, then 
\begin{eqnarray*}
I_b (i,j) & = & [b^i, b^{i+j}[_{\N};\\
I_b (i, + \infty) & = & [b^i, + \infty[_{\N}.
\end{eqnarray*} 
\end{definition}

The strings $(x_{n-1} x_{n-2} \dots x_0)_b$ with $x_{n-1} \not = 0$, namely the base-$b$ strings of length $n$,  describe all positive integers in $[b^{n-1}, b^n[_{\N}$. Therefore, any $b$-dc-semigroup $G$ can be expressed as
\begin{equation*}
G = \bigcup_{j \in J} I_b (j, 1),
\end{equation*}  
for some index set $J \subseteq \N$.

After a preliminary section, where  some technical results are proved, in Section \ref{sec_scpi} we describe a procedure for constructing the smallest $b$-dc-semigroup of $(\N^*, \cdot)$ which contains a set of positive integers. This latter problem was raised and solved in \cite{ros}, for $b=10$, relying upon a class of numerical semigroups \cite{ros_ns}, called the $LD$-semigroups. In the present paper we would like to present an alternative approach to \cite{ros}, removing also the restriction to the base we choose to represent the integers.

\section{Preliminaries}\label{sec_prel}
We prove a preliminary  result for  $b$-dc-semigroups, where $b > 2$. The analogue result for $2$-dc-semigroups will be stated immediately afterwards. 

\begin{lemma}\label{prel_b}
Let $G$ be a $b$-dc-semigroup, for some integer $b > 2$. 
\begin{enumerate}
\item The set $U_b = \{I_b (i, k): i \in \N, k \in \N^* \}$ endowed with the operation
\begin{displaymath}
I_b (i, k) \cdot I_b (j, l) = I_b (i + j, k + l)
\end{displaymath}
is an abelian semigroup.
\item If $\{i, j \} \subseteq \N$, $\{k, l \} \subseteq \N^*$ and
\begin{eqnarray*}
I_b (i, k) & \subseteq & G,\\
I_b (j, l) & \subseteq & G,
\end{eqnarray*}
then 
\begin{eqnarray*}
I_b (i, k) \cdot I_b (j, l) \subseteq  G.
\end{eqnarray*}
\end{enumerate}
\end{lemma}
\begin{proof}
\begin{enumerate}[leftmargin=*]
\item The claim holds since $(\N, +) \times (\N^*, +)$ is a direct product of abelian semigroups and multiplying $I_b(i, k)$ by $I_b(j,l)$ amounts to adding componentwise the two non-negative integers upon which they depend.
 
\item We notice that, for any $\tilde{k} \in [0, k[_{\N}$ and any $\tilde{l} \in [0, l[_{\N}$, 
\begin{displaymath}
\begin{array}{l}
b^{i+ \tilde{k}} \in I_b (i, k),\\
b^{j+ \tilde{l}} \in I_b (j, l),
\end{array}
\end{displaymath}
and consequently
\begin{displaymath}
\begin{array}{l}
b^{i+j+\tilde{k}+\tilde{l}} \in I_b(i+j,k+l).
\end{array}
\end{displaymath}
Moreover, as $\tilde{k}$ varies in $[0, k[_{\N}$ and $\tilde{l}$ varies in $[0, l[_{\N}$, we get all the powers $b^{i+j+h}$, for $h \in [0,k+l-1[_{\N}$. Therefore, $I_b(i+j,k+l-1) \subseteq G$.

Now we prove that also $b^{i+j+k+l-1} \in G$. By so doing, we can conclude that $I_b(i+j,k+l) = I_b(i,k) \cdot I_b(j,l) \subseteq G$. 

We notice that if $b$ is odd or $b$ is even and $i + k \geq 2$, then
\begin{equation*}
c_{ik} = \dfrac{b^{i+k-1} + b^{i+k}}{2} \in I_b(i, k).
\end{equation*}
In a similar way, if $b$ is odd or $b$ is even and $j + l \geq 2$, then
\begin{equation*}
c_{jl} = \dfrac{b^{j+l-1} + b^{j+l}}{2} \in I_b(j, l).
\end{equation*}
Suppose that one of the following holds:
\begin{itemize}
\item $b$ is odd;
\item $b$ is even, $i + k \geq 2$ and $j + l \geq 2$.
\end{itemize}
Then, $c_{ik} \cdot c_{jl} \in G$.
Moreover, by the AM-GM inequality, we have that
\begin{equation*}
c_{ik} \cdot c_{jl} \geq \sqrt{b^{2i+2k-1}} \cdot \sqrt{b^{2j+2l-1}} = b^{i+j+k+l-1}.
\end{equation*}
Since $G$ is closed with respect to the number of digits, we conclude that $[b^{i+j+k+l-1}, b^{i+j+k+l}[_{\N} \subseteq G$.

To end with, we consider the remaining cases.
\begin{itemize}
\item $b$ is even, $i+k=1$ and $j + l \geq 2$. We notice that $i+k=1$ if and only if $i=0$ and $k=1$. 
Since $b \not = 2$ we have that $b \geq 4$, $\sqrt{b} \geq 2$ and $\frac{b}{2} \in I_b (0, 1)$. Therefore, $\frac{b}{2} \cdot c_{jl} \in G$ and
\begin{equation*}
\dfrac{b}{2} \cdot c_{jl} \geq \dfrac{b}{\sqrt{b}} \cdot \sqrt{b^{2j+2l-1}} = b^{j+l}. 
\end{equation*}
Since $b^{i+j+k+l-1} = b^{j+l}$, we get the result.
\item $b$ is even, $i+k \geq 2$ and $j + l = 1$. The result follows by means of the same argument as above, replacing $j$ with $i$ and $l$ with $k$ respectively.
\item $b$ is even, $i + k = 1$ and $j + l =1$. In this case, $i=j=0$ and $k=l=1$. Since $\frac{b}{2} \in I_b (0, 1)$ we have that $\frac{b}{2} \cdot \frac{b}{2} \in G$. Moreover,
\begin{equation*}
\dfrac{b}{2} \cdot \dfrac{b}{2} \geq \dfrac{b}{\sqrt{b}} \cdot \dfrac{b}{\sqrt{b}} = b.
\end{equation*}
Since $b^{i+j+k+l-1} = b$, we get the result.
\end{itemize}
\end{enumerate}
\end{proof}

\begin{lemma}\label{prel_2}
Let $G$ be a $2$-dc-semigroup.   
\begin{enumerate}
\item The set $U_2 = \{I_2 (i, k): i \in \N, k \in \N^* \}$ endowed with the operation
\begin{displaymath}
I_2 (i, k) \cdot I_2 (j, l) =
\begin{cases}
I_2 (j, l) & \text{if $i = 0$ and $k = 1$,}\\
I_2 (i, k) & \text{if $j = 0$ and $l = 1$,}\\
I_2 (i + j, k + l) & \text{otherwise,}
\end{cases}
\end{displaymath}
is an abelian semigroup.
\item If $\{i, j \} \subseteq \N$, $\{k, l \} \subseteq \N^*$ and
\begin{eqnarray*}
I_2 (i, k) & \subseteq & G,\\
I_2 (j, l) & \subseteq & G,
\end{eqnarray*}
then 
\begin{eqnarray*}
I_2 (i, k) \cdot I_2 (j, l) \subseteq  G.
\end{eqnarray*}
\end{enumerate}
\end{lemma}
\begin{proof}
\begin{enumerate}[leftmargin=*]
\item Let $\{i_1, i_2, i_3 \} \subseteq \N$ and $\{k_1, k_2, k_3 \} \subseteq \N^*$. 

If $i_1 = 0$ and $k_1 = 1$, then
\begin{eqnarray*}
(I_2 (0, 1) \cdot I_2 (i_2, k_2) ) \cdot I_2 (i_3, k_3) & = & I_2(i_2, k_2) \cdot I_2 (i_3, k_3).
\end{eqnarray*}
Now we notice that $I_2(i_2, k_2) \cdot I_2 (i_3, k_3) = I_2 (i, j)$ for some $i \in \N$ and $j \in \N^*$. Therefore, 
\begin{eqnarray*}
I_2(i_2, k_2) \cdot I_2 (i_3, k_3) &  = & I_2 (i, j) = I_2(0,1) \cdot I_2 (i, j) \\
&  = & I_2(0,1) \cdot (I_2(i_2, k_2) \cdot I_2 (i_3, k_3)).
\end{eqnarray*}
In a similar way we can prove that the associativity of $\cdot$ holds if $i_2= 0$ and $k_2 = 1$ or $i_3 = 0$ and $k_3=1$.

Now we notice that, if $a \in \N$ and $b \in \N^*$, then
\begin{equation*}
I_2 (0,1) \cdot I_2(a, b) = I_2 (a, b) = I_2 (a, b) \cdot I_2 (0,1).
\end{equation*}
Hence, the operation $\cdot$ is commutative if $i=0$ and $k=1$ or $j=0$ and $l=1$. 

In the remaining cases, both the associativity and the commutativity of the operation $\cdot$ follow in analogy with Lemma \ref{prel_b}(1).
\item If $i=0$ and $k=1$ or $j=0$ and $l=1$ we get trivially the result. 

In the remaining cases, $i+k \geq 2$ and $j+l \geq 2$. Therefore, we can prove that $I_2(i+j, k+l) \subseteq G$ as in the proof of Lemma \ref{prel_b}(2), replacing each occurrence of $b$ with $2$.
\end{enumerate}
\end{proof}

Finally, we prove a technical lemma.
\begin{lemma}\label{l_struc}
Let $G$ be a $b$-dc-semigroup and $I_b(j,l) \subseteq G$ for some $j \in \N$ and $l \in \N^*$. Define $d = \lceil \frac{j}{l} \rceil$.
\begin{enumerate}
\item If $d \geq 1$, then $I_b(dj, + \infty) \subseteq G$.
\item If $d = 0$, $l \geq 2$ and $b =2$, then $G = \N^*$. 
\item If $d = 0$ and $b > 2$, then $G= \N^*$.
\end{enumerate}
\end{lemma}
\begin{proof}
We prove separately the claims.
\begin{enumerate}[leftmargin=*]
\item Let $x \in I_b(dj, + \infty)$ and $y = \lfloor \log_b(x) \rfloor$.

Suppose that
\begin{displaymath}
\begin{cases}
y = j q + r\\
0 \leq r < j
\end{cases}
\end{displaymath}
for some integers $r$ and $q$. 

Since $y \geq dj$, we have that $q \geq d$. All considered,
\begin{equation*}
x \in [b^{jq}, b^{jq+r+1}[_{\N} \subseteq I_b^q (j, l) =  [b^{jq}, b^{jq+ql}[_{\N} \subseteq G.
\end{equation*}
In fact,
\begin{equation*}
r+1 \leq j = \frac{j}{l} \cdot l \leq \left\lceil \frac{j}{l} \right\rceil \cdot l = dl \leq ql.
\end{equation*}

\item Since $d=0$, we have that $j = 0$ and $1 \in G$. Moreover, $[2^1, 2^2[_{\N} \subseteq I_2 (j, l)$. Therefore, $I_2(1,1) \subseteq G$. From (1) we deduce that $I_2(1, + \infty) \subseteq G$. Therefore $\{ 1 \} \cup [2, + \infty[_{\N} \subseteq G$, namely $G= \N^*$.

\item The claim holds since $I_b(0,1) \subseteq G$ and, if $x \in \N^* \backslash \{ 1 \}$ and $y =\lfloor \log_b(x) \rfloor$, then
\begin{equation*}
x \in [1, b^{y+1}[_{\N} = I_b^{y+1} (0,1) \subseteq G.
\end{equation*}
\end{enumerate}
\end{proof}

\section{The smallest $b$-dc-semigroup containing a set of positive integers}\label{sec_scpi}
In this  section we would like to address and generalize, through an alternative approach, the question raised in \cite{ros}, namely the problem of finding the smallest $b$-dc-semigroup containing a set of positive integers.

Consider a set $X \subseteq \N^*$, a base $b$ and the set
\begin{equation*}
J = \{\lfloor \log_b(x) \rfloor : x \in X \}.
\end{equation*} 

The semigroup
\begin{equation*}
G = \langle I_b (j, 1) : j \in J \rangle \subseteq U_b
\end{equation*}  
is the smallest $b$-dc-semigroup containing $X$. 

We analyse the possible cases more in detail, based upon the value of $j_0$, where 
\begin{equation*}
j_0 := \min \{j: j \in J \}.
\end{equation*}
\begin{itemize}     
\item \emph{Case 1:} $j_0 > 0$. Suppose that $[j_0, j_0+l_0[_{\N} \subseteq J$, for some positive integer $l_0$, namely $I_b(j_0, l_0) \subseteq G$. 

Let $d:=\lceil \frac{j_0}{l_0} \rceil$. 

\begin{itemize}
\item If $d=1$, then 
\begin{equation*}
I_b (j_0, + \infty) \subseteq G
\end{equation*}
in accordance with Lemma \ref{l_struc}(1). Moreover, since $\lfloor \log_b(x) \rfloor \geq j_0$ for any $x \in G$, we can conclude that $G = I_b (j_0, + \infty)$.

\item If $d > 1$, then
\begin{equation*}
I_b (d j_0, + \infty) \subseteq G
\end{equation*}
in accordance with Lemma \ref{l_struc}(1).

If we define $t:= d j_0$, then we can write 
\begin{equation*}
J \backslash [t, + \infty[_{\N} = \bigcup_{k=0}^n [j_k,j_k+l_k[_{\N},
\end{equation*}
for some non-negative integer $n$ and some positive integers $j_k$ and $l_k$ indexed on $K=[0,n+1[_{\N}$ such that
\begin{equation*}
j_{k+1} > j_k+l_k
\end{equation*}
for any $k \in K \backslash \{ n \}$.
All considered, $G$ is the union of $I_b(t, + \infty)$ with  all  possible products  
\begin{equation*}
\prod_{i=1 \dots e \atop j_{k_1} \leq \dots \leq j_{k_{e}}} I_b (j_{k_i}, l_{k_i})
\end{equation*} 
for any integer $e$ such that $1 \leq e < d$. Indeed, for any $e \geq d$ the products we obtain are contained in $I_b (t, + \infty)$.
\end{itemize}

\item \emph{Case 2:} $j_0 = 0$ and $b=2$. If $X = \{ 1 \}$, then $G = \{ 1 \}$. On the contrary, we define
\begin{equation*}
J^* := J \backslash \{ 0 \} 
\end{equation*}
and apply case 1 setting $J:=J^*$. The union of $\{ 1 \}$ with the semigroup we find applying the procedure described in case 1 is the smallest $2$-dc-semigroup containing $X$. 

\item \emph{Case 3:} $j_0 = 0$ and $b > 2$. In this case $G = \N^*$, according to Lemma \ref{l_struc}(3).  
\end{itemize}

\begin{example}\label{exm_10}
In \cite[Example 20]{ros} the authors dealt with the construction of the smallest $10$-semigroup $G$ containing $X = \{1235, 54321 \}$. We deal with the same problem proceeding as explained above.

First, we notice that $J = \{ 3, 4 \}$. Adopting the notations of case 1, we have that $j_0 = 3$ and $l_0=2$. Consequently, $d=2$. 

All considered, $I_{10} (6, + \infty) \subseteq G$ and $G = I_{10} (3,5) \cup I_{10} (6, + \infty)$.
\end{example}
\begin{remark}
We notice that if $b$ is an integer not smaller than $2$ and $X$ is a set formed by some positive integers, whose base-$b$ representations have length $4$ and $5$ respectively, then the smallest $b$-dc-semigroup $G$ containing $X$ is $I_{b} (3,5) \cup I_{b} (6, + \infty)$.

Indeed, it suffices to replace all the occurrences of $10$ with $b$ in Example \ref{exm_10} in order to get the result. 
\end{remark}
\bibliography{Refs}
\end{document}